\documentclass[a4paper,12pt, tikz]{article}
\usepackage[T1]{fontenc}
\usepackage[utf8]{inputenc}
\usepackage{lmodern}
\usepackage{xcolor}
\usepackage{textcomp}
\usepackage{graphicx}
\usepackage{multicol}
\usepackage{etoolbox}
\usepackage{picinpar}
\usepackage[utf8]{inputenc}
\usepackage{lmodern}
\usepackage{setspace}
\usepackage{everyshi}
\usepackage{pdfpages}
\usepackage[hyphens]{url}
\usepackage{hyperref}
\usepackage{pdfpages}
\usepackage{pgfpages}
\usepackage{amssymb}
\usepackage{mathtools}
\usepackage{amsmath}
\usepackage{amsthm}
\usepackage[notextcomp]{stix}
 \usepackage{paracol}
 \usepackage{stmaryrd}
 \usepackage{xurl}
 \usepackage{wasysym}
\usepackage{tikz}
 \usetikzlibrary{cd}
 \usepackage{authblk}
\usepackage[left=2cm, right=2cm, top=2cm, bottom=2cm]{geometry}

\renewcommand{\S}{\mathbb{S}^}

\newcommand{\Cech}{\operatorname{\check{C}ech}}

\newcommand{\lb}{\left(}
\newcommand{\rb}{\right)}
\newcommand{\lc}{\left\lbrace}
\newcommand{\rc}{\right\rbrace}
\newcommand{\ls}{\left[}
\newcommand{\rs}{\right]}
\renewcommand{\la}{\left\langle}
\renewcommand{\ra}{\right\rangle}

\newcommand{\dash}{^\prime}

\newcommand{\naturals}{\mathbb{N}}

\newcommand{\covr}{\operatorname{cov}}
\newcommand{\cov}{\operatorname{Cov}}
\newcommand{\link}{\operatorname{lk}}
\newcommand{\del}{\operatorname{del}}

\newcommand{\D}{\Delta}

\newcommand{\g}{\gamma}
\renewcommand{\t}{\tau}
\renewcommand{\a}{\alpha}

\newcommand{\s}{\sigma}
\newcommand{\p}{\pi}

\renewcommand{\implies}{\Rightarrow}

\newcommand{\cF}{\mathcal{F}}

\renewcommand{\H}{\mathcal{H}}

\newcommand{\N}{\mathcal{N}}

\newcommand{\tends}{\rightarrow}
\newcommand{\id}{\text{id}}
\newcommand{\inclusion}{\hookrightarrow}

\newcommand{\lmid}{\left\vert}
\newcommand{\rmid}{\right\vert}
\newcommand{\given}{~\middle\vert{}~}

\newtheorem{thm}{Theorem}[section]

\newtheorem{defn}[thm]{Definition}

\newtheorem{lem}[thm]{Lemma}
\newtheorem{conj}[thm]{Conjecture}

\newtheorem{note}[thm]{Note}

\usepackage{tikzpagenodes}
\usepackage{tikz}
\usepackage{eso-pic}
\begin{document}
\title{$\Cech$ complexes for finite sets}
%\date{}
\author{Anamitro Biswas \thanks{The work was done when the author was affiliated to Department of Mathematics, Indian Institute of Technology Bhilai. The problem was suggested by Dr. Anurag Singh, Department of Mathematics, IIT Bhilai. The author also gratefully acknowledges important insights received from Dr. Himanshu Chandrakar, who was at IIT Bhilai at that time.}}
%\affil[1]{Department of Mathematics, Indian Institute of Technology Bhilai}

\maketitle
\begin{abstract}
\noindent{}Let $\mathcal{F}^{m}=\{A\subset[m]\}$, endowed with the symmetric difference metric and the Cardinality Reverse-Lexicographic order. That way, we compute the homotopy type $\Cech\lb \cF_{\preceq a}^m; \frac{r}{2}=1\rb$ by relating it to consecutively adding vertices of $\mathbb{I}^m$, the unit hypercube of $m$th dimension, to point $0$. Further, for higher $r$, we propose a conjecture regarding contractibility of a subcomplex implying a formula for the exact number of wedges of $\S 3$'s and $\S 4$'s in the $\Cech$ complex.
\end{abstract}
\tableofcontents
\section{Introduction}
Let $\mathcal{F}_{n}^{m}=\{A\subset[m]:|A|=n\}$. We equip this space with the symmetric difference metric $d$; $d(A,B) = |(A\setminus B)\cup(B\setminus A)|$. For $S$ a finite subset of a metric space $X$, and a scaler $r\ge0$, the \v{C}ech complex $\Cech(S;r)$ is an abstract simplicial complex whose vertex set is $S$ and $\sigma\subset S$ is a simplex if and only if $ \bigcap_{x\in\sigma}B\left[x,\frac{r}{2}\right]\ne\phi $ where $B[x,\frac{r}{2}]$ is a closed ball of radius $\frac{r}{2}$ around $x$. We know that a subset of a simplex is also a simplex.

If we have two subsets $a,b\subset \ls m\rs$, i.e., $a,b\in P\lb \ls m\rs\rb$, we arrange them in the ascending order in the following way:
\begin{itemize}
  \item If $|a|>|b|$, we also have $a\succ  b$, where $|a|$ is the cardinality of $a$.
  \item If $|a|=|b|$, then let $a=\lc i_1,i_2,\dots, i_{|a|}\rc$ and $b=\lc j_1,j_2,\dots, j_{|a|}\rc$. Let $k$ be the smallest index such that $i_k\neq j_k$. Then $a\prec  b$ accordingly as $i_k<  j_k$ etc.
\end{itemize}
We can see this in an alternative manner. See each $a=\lc i_1,\dots i_{|a|}\rc\in P\lb\ls m\rs\rb$ as an element in the $m$-dimensional unit hypercube $\mathbb{I}^m$, with each $i_k$th coordinate 1 for $k\in\lc 1,\dots, |a|\rc$ and all other coordinates 0. In this scenario, our ordering is equivalent to:
\begin{itemize}
  \item If $|a-0|>|b-0|$, we also have $a\succ b$, where $|a-0|$ is the Hamming distance of $a$ from 0, i.e., $|a-0|=\sum_{i=1}^m\lmid a_i-0\rmid$.
  \item Between points with the same distance from $\lb 0, \dots , 0\rb$, we follow the reverse of lexicographic order.
\end{itemize}
We call this the \emph{Cardinality Reverse-Lexicographic (CRL) order}. Also, observe that with this notation, if we drop the subsequent 0's after the last non-zero coordinates, we can view our point in $\mathbb{I}^m$ for a lesser $m$ (with easier visualization), but that does not affect our calculations. Because coordinates in higher dimension don't bother us here, and also we can carry forward our results to higher dimension without any change, just by reverting to coordinates with the latter entries ``0''.

We wish to compute $\Cech \lb\cF_n^m;r\rb,~r>0$. When we are looking at $\N\lb \cF_2^m;2\rb$, the facets are $N\ls \lc x,y\rc\rs=\lc\lc i,x\rc~:~i\in\ls m\rs\setminus x\rc\cup \lc \lc j,y\rc~:~j\in\ls m\rs\setminus y\rc$. In other words, $\s\in E\lb K_m\rb$ forms a simplex is $\exists~x,y$ such that $\forall~e\subset \s$ either $x\in e$ or $y\in e$ or both. Given $H\in E(G)$ we say that a subset $A\subseteq V(G)$ covers $H$ is every edge in $H$ intersects $A$. Covering number of $H$ is defined as $\displaystyle\min_{A\subseteq G}\lc \# A: A\text{ is a cover of }H\rc$. Also, $\s\in\N\lb \cF_2^m;2\rb$ if and only if $\covr\lb\s\rb\leq 2$. In \cite[Section 26.6]{J}, $\cov_{m,2}$ corresponds to our $\N\lb \cF_2^m;4\leq r<8\rb$, which should be homotopic to a wedge of $\S 3$'s and $\S 4$'s. However, the exact number of wedge components are not known, and is worth investigation.

\cite{FN} computes homotopy types for Vietoris-Rips complexes of such spaces. In this paper, we exhaustibly describe the open balls in $\Cech\lb \cF_n^m; r\rb$ and in the subsequent section, using computations for smaller $n$, inductively prove a formula for the homotopy types of $\Cech\lb \cF_n^m; r=2\rb$ and $\Cech\lb \cF_{\preccurlyeq a}^m; r=2\rb$ for some $a\in\cF_n^m$. In the last section, we talk about future directions, and show an approach to compute the homotopy types for general $r$, calculating the exact number of $\S 3$'s and $\S 4$'s in the wedge.

\section{Preliminaries}
\subsection{Simplicial complexes}
A \emph{simplicial complex} $L$ on a finite set $V$ is a family of non-empty subsets of $V$ such that if $\s\in L$ and $\t\subseteq \s$ is non-empty, then $\t\in L$.
\begin{defn}[Nerve complex]
Let $\mathcal{U}=\lc U_i\rc_{i\in\lc 1, \dots, k\rc}$ be a collection of subsets of a metric space X. The nerve of $\mathcal{U}$, $\N\lb \mathcal{U}\rb$ is a simplicial complex such that
\begin{itemize}
  \item the vertex set is $\mathcal{U}$;
  \item a subset $\s\subseteq \mathcal{U}$ is a simplex if and only if $\cap_{U\in \s}U\neq \phi$.
\end{itemize}
\end{defn}
A $\Cech$ complex is the nerve of the corresponding collection of $r$-balls, i.e., $\Cech\lb S; r\rb=\N\lb\lc B\lb s,r\rb\rc_{s\in S}\rb$.
\subsection{The link-deletion method}
For any vertex $v$ in a complex $K$, \emph{deletion} $K\setminus v$ denotes the induced complex on the vertex set $K^{(0)}\setminus \lc v\rc$. The \emph{link} of $v$ in $K$, $\link_K\lb v\rb=\lc \s:\s\cup\lc v\rc\in K\text{ and }v\not\in\s\rc$.

We have the following lemma as an useful computational tool:
\begin{lem}\label{k1-union-k2}
\cite{GSS} If $K$ is a simplicial complex and for $K=K_1\cup K_2$, the inclusion maps $i_j:K_1\cap K_2\inclusion K_j$ are both null-homotopic. Then $K\simeq K_1\vee K_2\vee \Sigma\lb K_1\cap K_2\rb$.
\end{lem}
As a corollary, we have,
\begin{lem}\label{link-deletion}
\cite[Lemma 2]{AA}
Let $v\in K$ be a vertex and the inclusion $\link_K\lb v\rb\inclusion K$ is null-homotopic. Then, $K\simeq K\setminus v\vee\Sigma\lb \link_K(v)\rb$.
\end{lem}

\subsection{Homotopy type of a complete graph}
\begin{defn}
A \emph{complete graph} is a simple undirected graph in which every pair of distinct vertices has a unique edge between them. The complete graph for $n$ vertices, $n\in\naturals$, is denoted by $K_n$.
\end{defn}
Any connected topological graph is homotopic to a wedge sum of $\S 1$'s, the number of which is given by the formula $\#\text{edges}-\#\text{vertices}+1$. So,
$$\displaystyle K_n\simeq \bigvee_{\lmid E\lb K_n\rb\rmid-\lmid V\lb K_n\rb\rmid+1}\S 1=\bigvee_{\binom{m}{2}-m+1}\S 1=\bigvee_{\binom{m-1}{2}}\S 1.$$
Using this result, we shall show in Eq (\ref{use-of-nerve}) that the homotopy type of $n+1$ pairwise point-intersecting $n$-simplices is a wedge of $\S 1$'s. We also use the Nerve Theorem.

\begin{thm}\label{nerve}
[Nerve theorem]
Let $F$ be a finite collection of closed, convex sets in the Euclidean space. Then the nerve of $F$ and the union of the sets in $F$ have the same homotopy type.
\end{thm}

\section{Observations about $\Cech\lb \cF_n^m;r\rb$}

\subsection{The open balls}
\subsubsection{The case $n=1$}
For $n=1$, the elements of $\cF_{1}^{m}$ are singletons $\{x\}$ for $x\in[m]$. Therefore, $d(\{x\},\{y\})=2$. We have that all singletons of singletons belong to \v{C}ech $(\cF_{n}^{m};r)$ since if $\sigma=\lc\lc x\rc\rc$, then $ \bigcap_{v\in\sigma}B\left[v,\dfrac{r}{2}\right]\ne\phi, $ since $v$ has only one option $\{x\}$ and the intersection is trivially a singleton. For $0\le\dfrac{r}{2}<2$ i.e., $0\leq r<4$, $B\ls v,\dfrac{r}{2}\rs=\{\{v\}\}$, while $B\ls v,\dfrac{r}{2}\ge2\rs=\{\{i\}:i\in\ls m\rs\}$, and hence $\la\{1\},\dots, \{m\}\ra$ is an $(m-1)$-simplex generated by all of the points.

\subsubsection{The case $n=2$}
For $n=2$, we first compute possible distances.
\begin{itemize}
  \item $d\lb \lc x,y\rc, \lc x,y\rc\rb=0$;
  \item $d\lb \lc x,y\rc, \lc y,z\rc\rb=2$;
  \item $d\lb \lc x,y\rc, \lc u,v\rc\rb=4$.
\end{itemize}
Let $v=\lb a, b\rb$ where $a<b\in\ls m\rs$. Then $B\ls v, \dfrac{r}{2}<2\rb=\lc v\rc$. Also,
$$B\ls v, 4>\dfrac{r}{2}\geq 2\rs=\lc\lc a,i\rc\lc j,b\rc~:~a\neq i\in\ls m\rs, b\neq j\in\ls m\rs\rc,$$
since at distance 0 lies the point $\lc a, b\rc$ itself at exactly distance 2 from $v$ lie those points whose one entry is $a$ and the other something else; and at distance 4 lies every other point in the space. Again, $B\ls v, \dfrac{r}{2}\geq 4\rs$ again consists of all the points in the space. For $\dfrac{r}{2}<4$, a simplex would be generated by $\lc\lc a, i\rc~:~a\neq i\in\ls m\rs\rc$ so that the intersection will be non-empty.

\subsubsection{The case $n=3$}
We proceed in a similar manner, with possible distances:
\begin{itemize}
  \item $d\lb \lc x,y,z\rc,\lc x,y,x\rc\rb=0$;
  \item $d\lb \lc x,y,z\rc,\lc x,y,w\rc\rb=2$;
  \item $d\lb \lc x,y,z\rc,\lc x,u,v\rc\rb=4$;
  \item $d\lb\lc x,y,z\rc, \lc u.v.w\rc\rb=6$.
\end{itemize}
Let $V=\lb a,b,c\rb$ where distinct $a,b,c\in\ls m\rs$.
\begin{itemize}
  \item $B\ls V, 0\leq \dfrac{r}{2}<2\rs=\lc v\rc$; hence each simplex is of the form $\s=\lc \lc V\rc\rc$ for $V\in\cF_3^m$.
  \item $B\lb V, 2\leq\dfrac{r}{2}<4\rb=\lc A\in\cF_3^m~:~\exists~ B\in \cF_2^m\ni B\subset V\text{ and }B\subset A\rc$; hence each simplex is of the form $N^{(1)}\ls i_1, i_2, i_3\rs=\lc A\in \cF_3^m~:~\lc i_1, i_2\rc\subset A\text{ or }\lc i_3, i_4\rc\subset A\rc$. Here $i_j\in\ls m\rs$ and $i_{j_1}\neq i_{j_2}$. In other words, $N^{(1)}\ls v\rs=\lc A\in \cF_3^m~:~\exists~ B\in\cF_2^m\ni B\subset V\text{ and }B\subset A\rc$.
  \item $B\lb V, 4\leq\dfrac{r}{2}<6\rb=\lc A\in\cF_3^m~:~\exists~ B\in \cF_1^m\ni B\subset V\text{ and }B\subset A\rc$; hence each simplex is of the form $N^{(2)}\ls i_1, i_2, i_3\rs=\lc A\in \cF_3^m~:~\lc i_1\rc\subset A\text{ or }\lc i_2\rc\subset A\text{ or }\lc i_3\rc\subset A\rc$. Here $i_j\in\ls m\rs$ and $i_{j_1}\neq i_{j_2}$. Hence $N^{(2)}\ls V\rs=B\ls V, 4\leq \dfrac{r}{2}<6\rs$.
  \item $B\ls V,\dfrac{r}{2}\geq 6\rs=\cF_3^m$. The whole space is a $\lb m^3-1\rb$-simplex.
\end{itemize}

\subsubsection{The generalization}
In general, for $V\in\cF_n^m$,
\begin{itemize}
  \item $B\ls V, \dfrac{r}{2}\rs~\lb 0\leq r<4\rb=\lc v\rc$; hence each simplex is of the form $\s=\lc \lc V\rc\rc$ for $V\in\cF_n^m$.
  \item For $4\leq r<8$, each simplex is of the form
  $$N^{(1)}\ls V\rs=\lc A\in \cF_n^m~:~\exists~\lc i_1, i_2,\dots, i_{n-1}\rc\subset V\ni\lc i_1, i_2,\dots i_{n-1}\rc\subset A\rc,$$ where $V\in \cF_n^m$.
  \item For $8\leq r<12$, each simplex is of the form
  $$N^{(2)}\ls V\rs=\lc A\in \cF_n^m~:~\exists~\lc i_1, i_2,\dots, i_{n-2}\rc\subset V\ni\lc i_1, i_2,\dots i_{n-2}\rc\subset A\rc.$$
  \item[]\vdots
  \item $B\lb V, r\geq 2n\rb=\cF_n^m$. the whole space is a $\lb m^n-1\rb$-simplex.
\end{itemize}
For $4k\leq r<4\lb k+1\rb$, each simplex is of the form
$$N^{(k)}\ls V\rs=\lc A\in \cF_n^m~:~\exists~\lc i_1, i_2,\dots, i_{n-k}\rc\subset V\ni\lc i_1, i_2,\dots i_{n-k}\rc\subset A\rc,$$ where $i_j\in\ls m\rs$ and $i_{j_1}\neq i_{j_2}$.
\section{Inductively, for $\cF_2^m$}
Clearly, $G_1$ consists of a single point $(0,0,0)$, hence $\Cech\lb G_1, \frac{r}{2}=1\rb$ is contractible. Similarly, $G_2$ and $G_3$ are contractible.
\begin{figure}[h]
\centering
\includegraphics[width=0.3\textwidth]{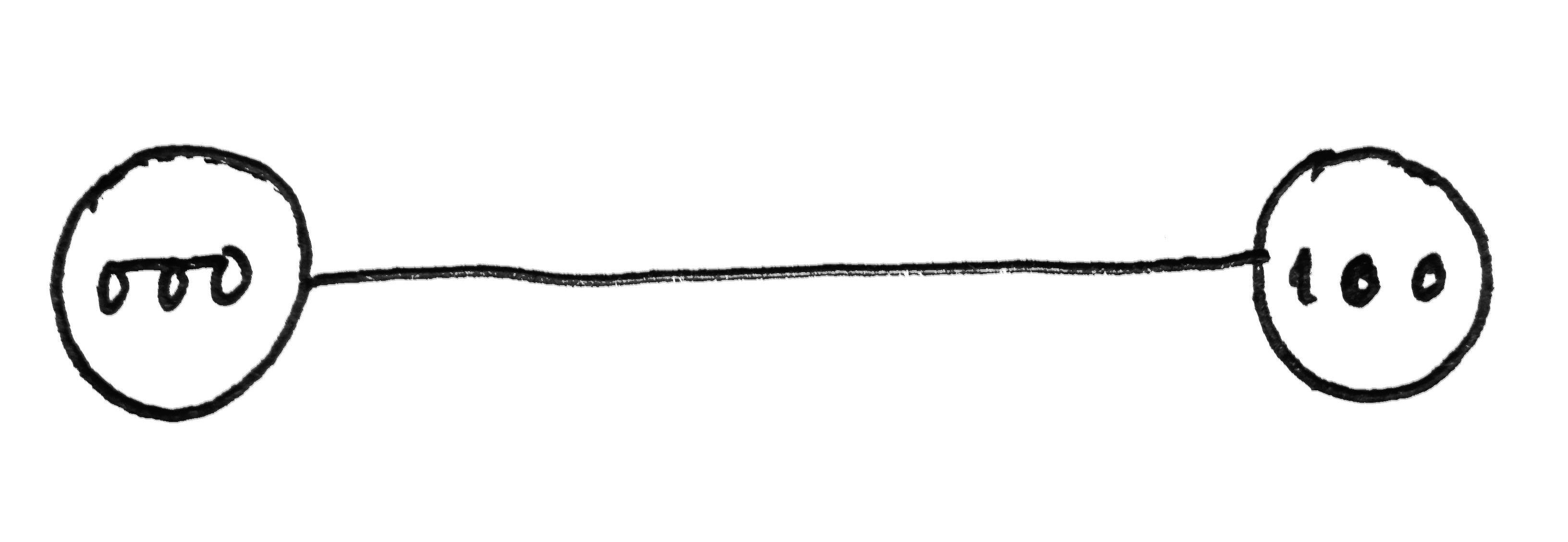}
\caption{$G_2$}
\label{fig:G2}
\end{figure}
\begin{figure}[h]
\centering
\includegraphics[width=0.3\textwidth]{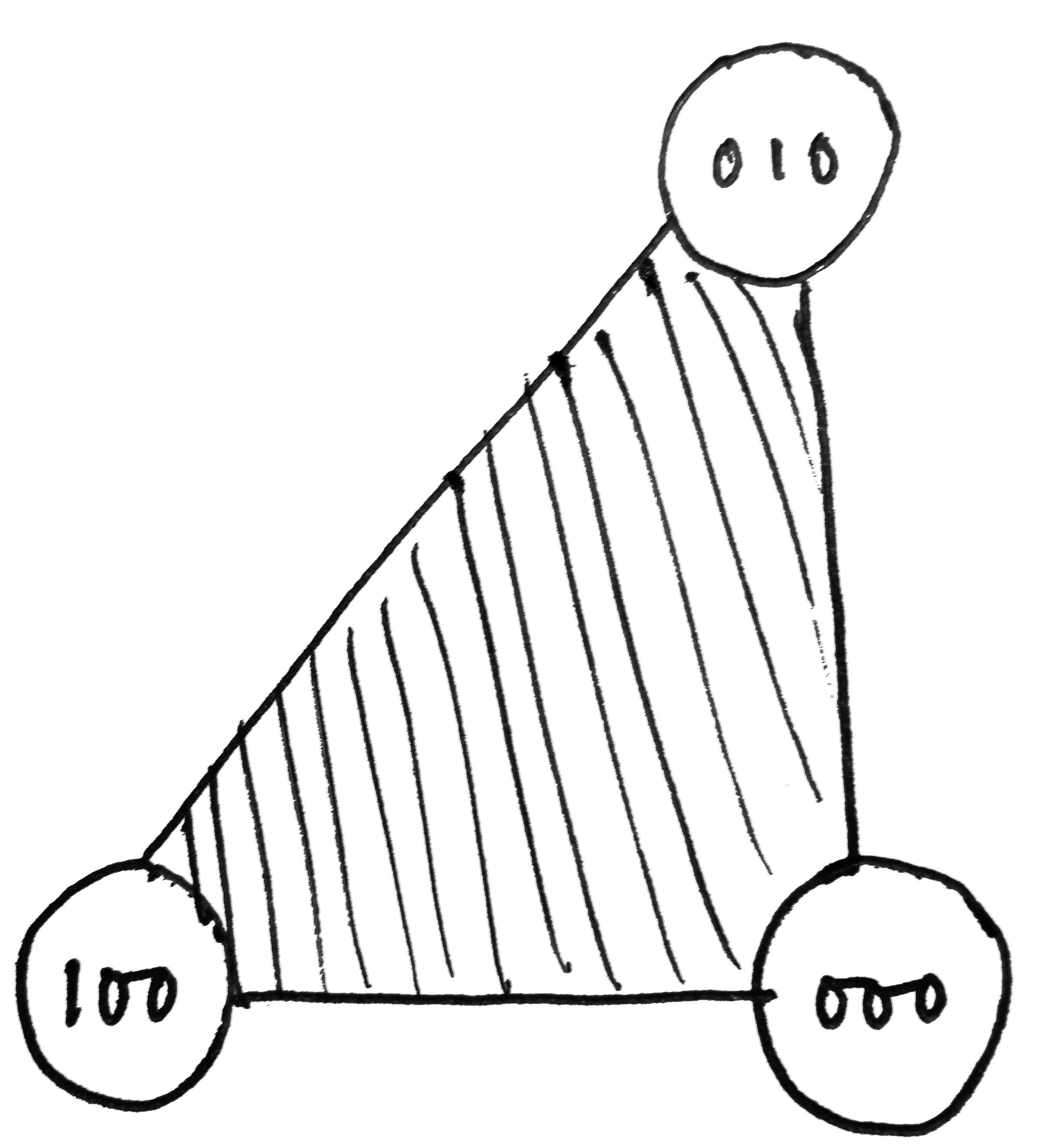}
\caption{$G_3$}
\label{fig:G3}
\end{figure}
\subsection{$G_4$}
As for $G_4$, we have 4 points: $(0,0,0), (1,0,0), (0,1,0), (0,0,1)$, and their neighbourhoods
\begin{itemize}
  \item $N\ls(0,0,0)\rs=\lc(0,0,0), (1,0,0), (0,1,0), (0,0,1)\rc$;
  \item $N\ls (1,0,0)\rs=\lc (1,0,0), (0,0,0)\rs$;
  \item $N\ls (0,1,0)\rs=\lc(0,1,0), (0,0,0)\rc$;
  \item $N\ls (0,0,1)\rs=\lc (0,0,1), (0,0,0)\rc$.
  \begin{figure}[h]
\centering
\includegraphics[width=0.3\textwidth]{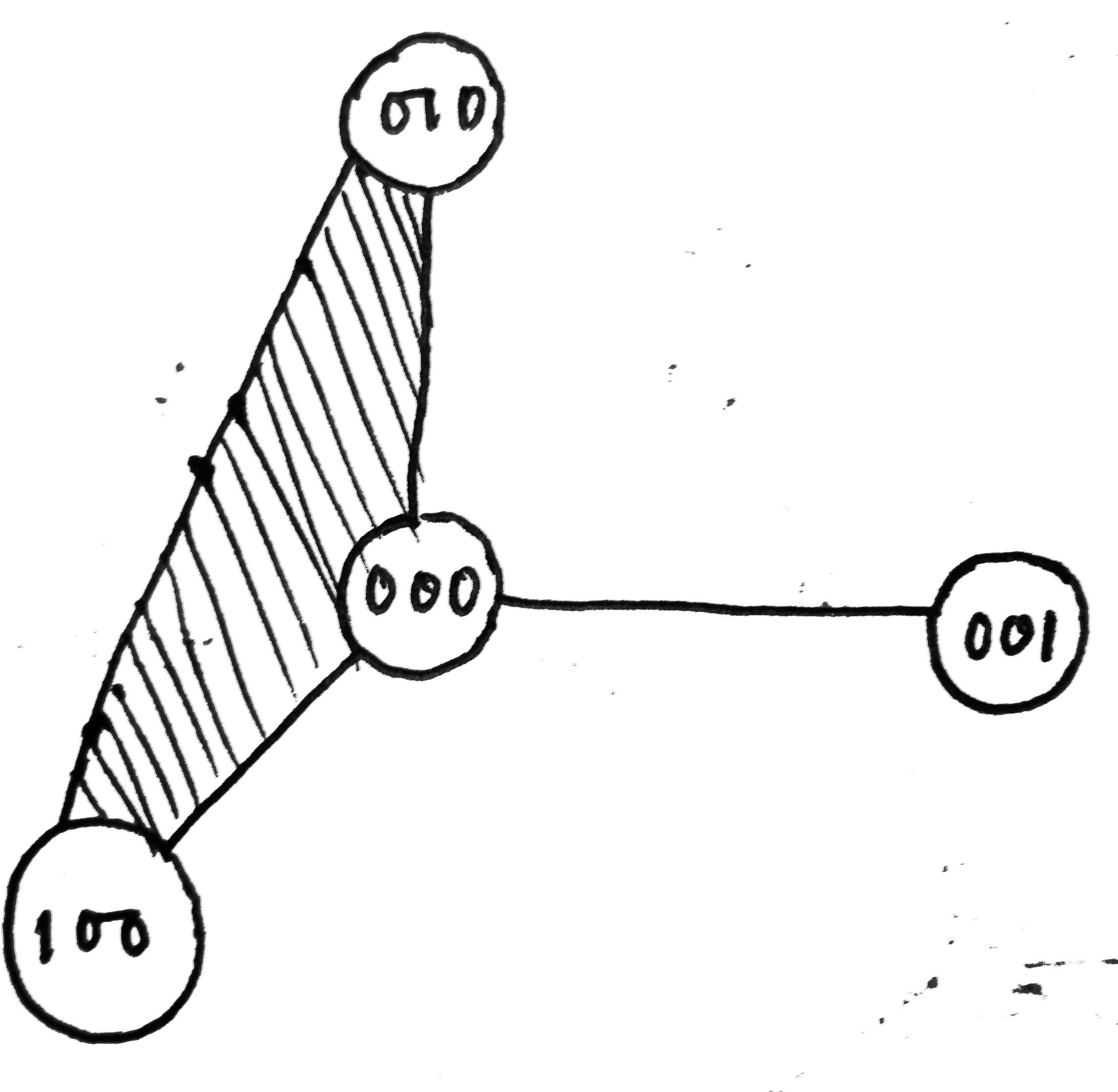}
\caption{$G_4$}
\label{fig:G4}
\end{figure}
  So, $G_4$ is contractible.
\end{itemize}
\subsection{$G_5$}
We have $\Cech\lb G_5\rb=\la N[(0,0,0)], N[(1,0,0)], N[(0,1,0)], N[(0,0,1)], N[(1,1,0)]\ra$.
\begin{itemize}
  \item $N[(0,0,0)]=\lc (0,0,0), (0,0,1), (1,0,0), (0,1,0)\rc$;
  \item $N[(1,0,0)]=\lc (1,0,0), (0,0,0), (1,1,0)\rc$;
  \item $N[(0,1,0)]=\lc (0,1,0), (1,1,0), (0,0,0)\rc$;
  \item $N[(0,0,1)]=\lc (0,0,1), (0,0,0)\rc$;
  \item $N[(1,1,0)]=\lc (1,1,0), (0,1,0), (1,0,0)\rc$.
\end{itemize}
\begin{figure}[h]
\centering
\includegraphics[width=0.45\textwidth]{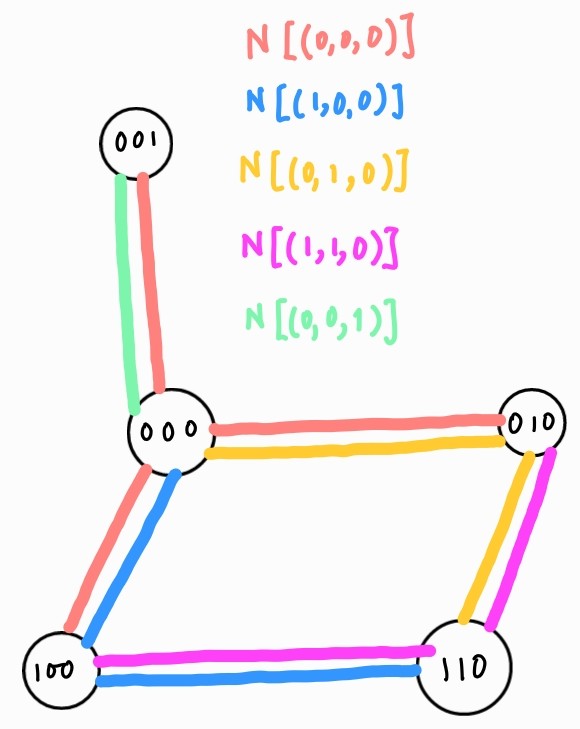}\includegraphics[width=0.45\textwidth]{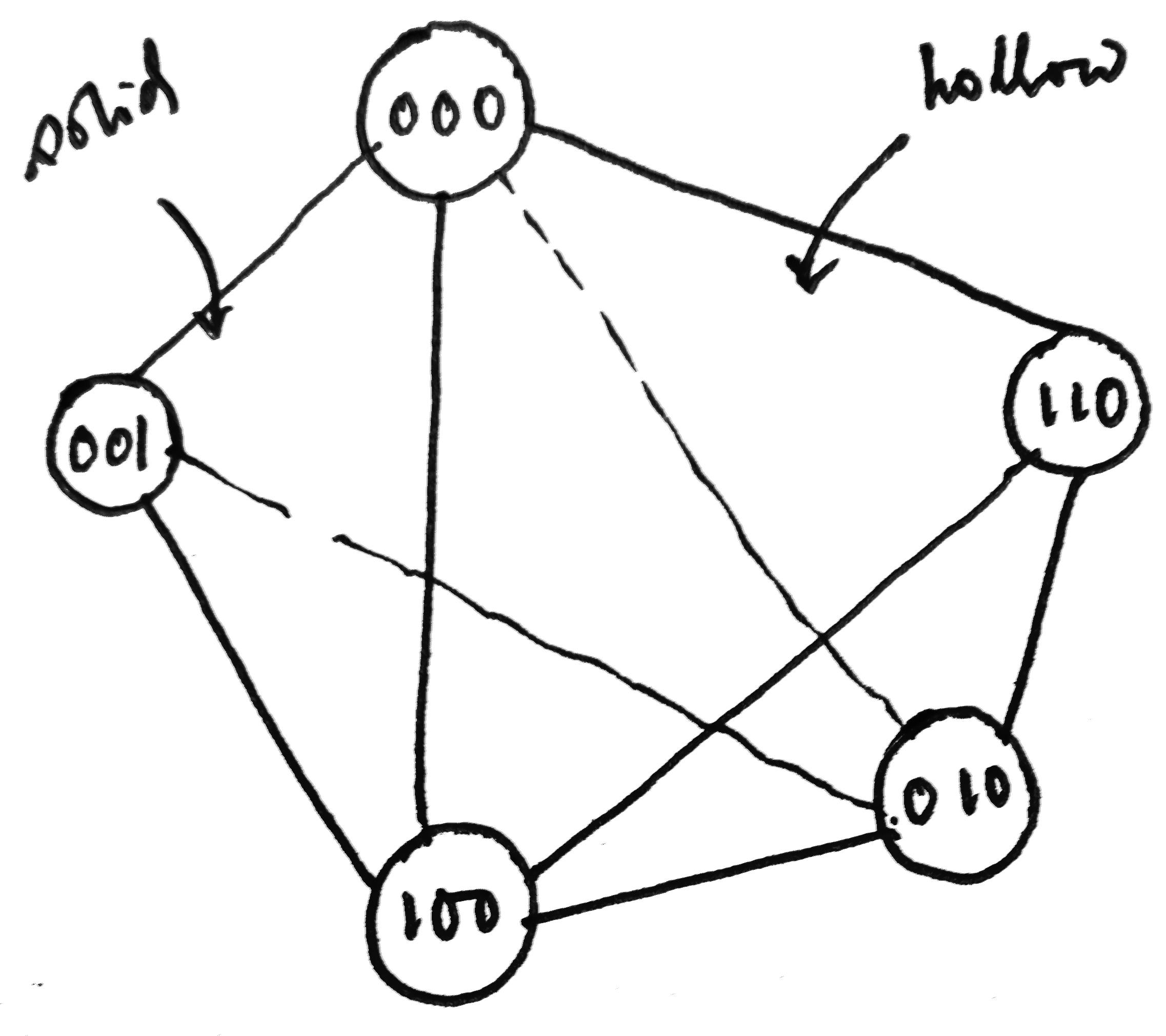}
\caption{$G_5$}
\label{fig:G5}
\end{figure}
Now, $\del\lb (1,1,0)\rb=\lc \s\in\Cech\lb G_5\rb:(1,1,0)\not\in\s\rc=\Cech\lb G_4\rb\simeq *$. Also, $\link\lb (1,1,0)\rb=\la N[v]:v\in N[(1,1,0)]\ra$\\$=\Delta\simeq \S 1$. By the link-deletion method, $\Cech\lb G_5\rb\simeq \S 2$.

For visualization, one can imagine filling in the discreteness, $G_5$ will have the structure of two tetrahedra, one hollow and one solid, attached at the surface $\lc (0,0,0), (1,0,0), (0,1,0)\rc$ [Fig \ref{fig:G5}].

This will be our general technique now on. We shall take the last-added point $v$, and in general,
\begin{equation}\label{deletion-induction}
  \Cech\lb G_k\rb\simeq \del\lb v\rb\vee \Sigma \link_{\Cech\lb G_k\rb}\lb v\rb=\Cech\lb G_{k-1}\rb\vee \Sigma \link_{\Cech\lb G_k\rb}\lb v\rb.
\end{equation}

\subsection{$G_6$}
We have $\Cech\lb G_6\rb=\la N[(0,0,0)],  N[(1,0,0)], N[(0,1,0)], N[(0,0,1)], N[(1,1,0)], N[(1,0,1)]\ra$, where
\begin{itemize}
  \item $N[(0,0,0)]=\lc (0,0,0), (0,0,1), (1,0,0), (0,1,0)\rc$;
  \item $N[(1,0,0)]=\lc (1,0,0), (0,0,0), (1,1,0), (1,0,1)\rc$;
  \item $N[(0,1,0)]=\lc (0,1,0), (1,1,0), (0,0,0)\rc$;
  \item $N[(0,0,1)]=\lc (0,0,1), (0,0,0), (1,0,1)\rc$;
  \item $N[(1,1,0)]=\lc (1,1,0), (0,1,0), (1,0,0)\rc$;
  \item $N[(1,0,1)]=\lc (1,0,0), (0,0,1), (1,0,1)\rc$.
\end{itemize}
In Figure \ref{fig:G6}, tetrahedra $\la (0,0,0), (0,0,1), (1,0,0), (0,1,0)\ra, \la (1,0,0), (0,0,0), (1,1,0), (1,0,1)\ra$ are solid, and the other two are hollow. Then, $\del(1,0,1)\simeq \Cech\lb G_5\rb\simeq \S 2$ and $\link(1,0,1)=\text{hollow }\la (0,0,1), (1,0,0),\right.$\\$\left. (0,0,0)\ra\simeq \S 1$. Then $\link(1,0,1)$ is contractible inside $\del(1,0,1)$. Thus, $\Cech\lb G_6\rb\simeq \S 2\vee \S 2$.
\begin{figure}[h]
\centering
\includegraphics[width=0.45\textwidth]{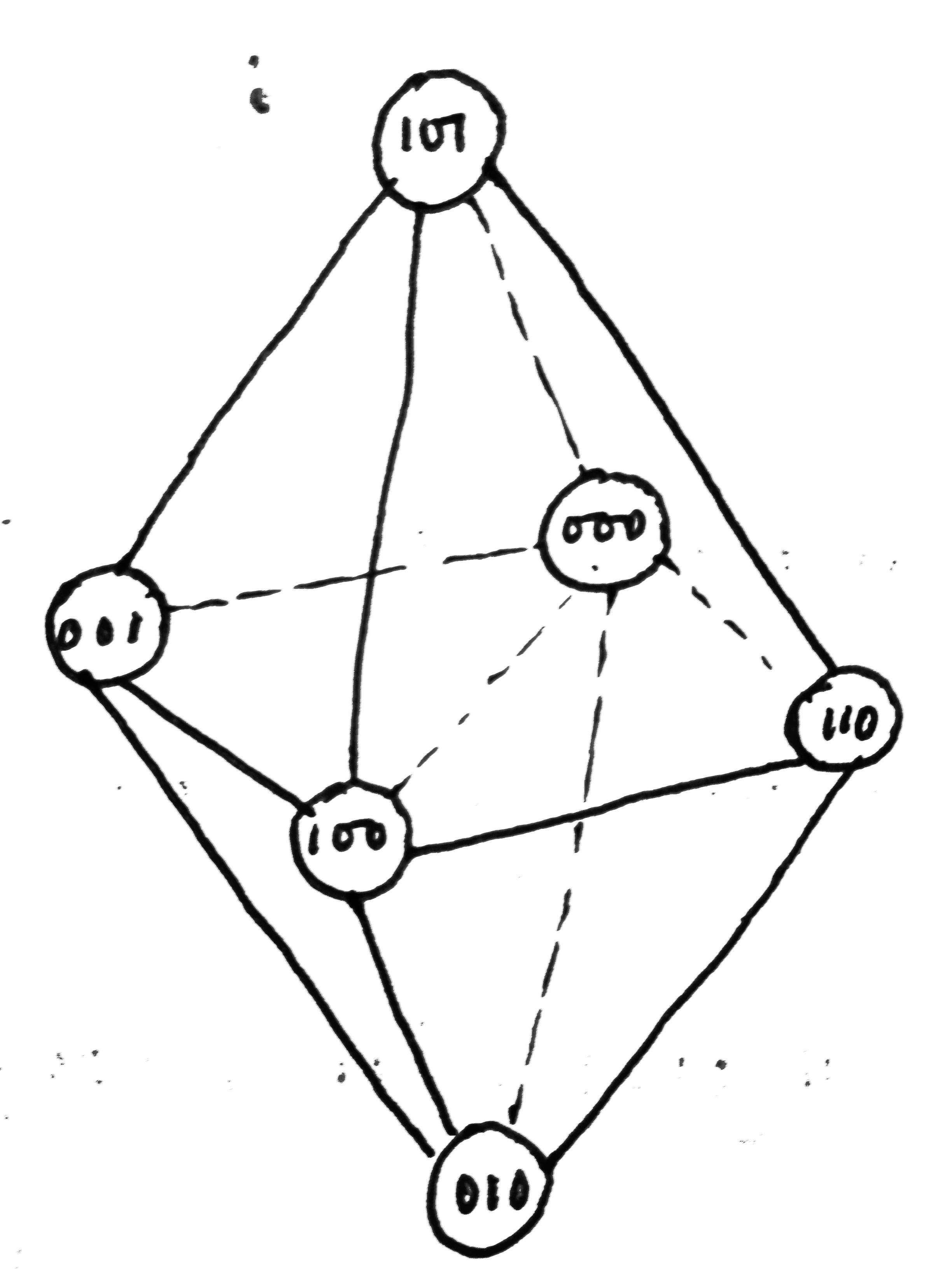}
\caption{$G_6$}
\label{fig:G6}
\end{figure}
 On similar calculations, we get $G_7\simeq \bigvee_4 \S 2$.
 
 \subsection{For sets of cardinality 2}
 Alternatively, we have the following lemma:
 \begin{lem}
 By moving ahead with one more vertex $\a$ in the CRL order, as long as $\lmid\a\rmid=2$, we wedge our previous $\Cech$ with an $\S 2$.
 \end{lem}
 \begin{proof}
 Let $\a$ be the last added vertex. Supposing $\Cech[\a]=\la \a, \a_1, \a_2\ra$, we have, since $\la\a_1,\a_2, 0\ra$ goes inside $\Cech[0]$, three facets containing $\a$:
 \begin{itemize}
   \item $\Cech\ls \a_1\rs=\la\a_1, \a, 0\ra$,
   \item $\Cech\ls \a_2\rs=\la\a_2, \a, 0\ra$,
   \item $\Cech[\a]$.
 \end{itemize}
 We know, $\link(\a)=\lc \t\in\Cech\lb G_{\text{no. }\a}\rb\given \a\not\in\t,~\t\cup \a\in\Cech\lb G_{\text{no. }\a}\rb\rc$, where $\text{no. }\a$ is the sequential numbering of graph when $\a$ is added. Removing $\a$ from the 3 new facets we get 3 lines: $\overline{\a_10}$, $\overline{\a_20}$, $\overline{\a_1\a_2}$. Of course this is not a filled-in triangle, since then $\la\a_1, \a_2,0\ra\subset \link(\a)\implies \la\a_1, \a_2,0, \a\ra\subset \Cech\lb G_{\text{no. }\a}\rb$, a contradiction. This $\link(\a)$ is contractible in $\Cech\lb G_{\text{no. }\a-1}\rb$, which is, by induction hypothesis, wedge of $\S 2$'s. Thus, by induction, using Eq (\ref{deletion-induction}), we have $\displaystyle\Cech\lb G_{\text{no. }\a}\rb\simeq\bigvee_{\text{no. }\a-4}\S 2$.
 \end{proof}
Thus, $\displaystyle\Cech\lb G_{\text{no. }\a}\rb\simeq \bigvee_{\binom{m}{2}}\S 2$ when $\displaystyle\a=\max\cF_2^m$.

\subsection{Generalization}
Let us move to the general case. We add a point $x\in\mathbb{I}_m$, and $x$ has $n$ non-zero coordinates. If there are any previous points with $n$ non-zero coordinates included in the complex, none are ``adjacent'' to $x$. So, basically, we get $(n+1)$ new facets:
\begin{itemize}
  \item $N[x]=\la x, x_1, x_2, \dots, x_n\ra$;
  \item $N[x_1]=\la x, x_1 x_1^{(1)}, x_1^{(2)}, \dots, x_1^{(n-1)}\ra$;
  \item $N[x_2]=\la x, x_2, x_2^{(1)}, x_2^{(2)}, \dots, x_2^{(n-1)}\ra$;
  \item[] $\vdots$
  \item $N[x_n]=\la x, x_n, x_n^{(1)}, x_n^{(2)}, \dots, x_n^{(n-1)}\ra$,
\end{itemize}
where $N[x]$ is the simplex around $x$, $x_i$ is the face of $x$ with the $i$th non-zero-coordinate made 0, $x_i^{(j)}$ is the face of $x_i$ with the $j$th non-zero-coordinate made 0. We have,
\begin{eqnarray*}
\link_{\Cech\lb \text{no. }x\rb}(x)&=&\lc\la x_1, x_2, \dots x_n\ra, \la x_1, x_1^{(1)}, x_1^{(2)}, \dots, x_1^{(n-1)}\ra, \la x_2, x_2^{(1)}, x_2^{(2)}, \dots, x_2^{(n-1)}\ra, \dots,\right.\\
&&\left. \la x_n, x_n^{(1)}, x_n^{(2)}, \dots, x_n^{(n-1)}\ra\rc.
\end{eqnarray*}
Transliteration to our ``subset of $[m]$'' notation,
\begin{eqnarray*}
\link_{\Cech\lb \text{no. }x\rb}(x)&=&\lc\la x\setminus\lc i_1\rc, x\setminus\lc i_2\rc, \dots x\setminus \lc i_n\rc\ra,\right.\\
&&\left.\la x\setminus\lc i_1\rc, x\setminus\lc i_1, i_2\rc, x\setminus\lc i_1, i_3\rc, \dots,x\setminus\lc i_1, i_n\rc\ra, \dots,\right.\\
&&\left.\la x\setminus\lc i_n, i_1\rc, x\setminus\lc i_n, i_2\rc, \dots, x\setminus\lc i_n, i_{n-1}\rc\ra\rc.
\end{eqnarray*}
These are $n$-simplices, $(n+1)$ in number, any 2 of which have exactly 1 point in common, and no two the same one. Irrespective of the dimension of the simplices, by the Nerve lemma, since all intersections are contractible, $\link_{\Cech\lb G_{\text{no. }x}\rb}(x)$ is homotopic to the union of the simplices, which in turn is homotopic to $K_{n+1}$, the complete graph with $n+1$ vertices. The homotopy type is then $\displaystyle\bigvee_{\binom{n}{2}}\S 1$. Since this is contractible inside $\del_{\Cech\lb G_{\text{no. }x}\rb}(x)$, a wedge of $\S 2$'s, using Eq \ref{deletion-induction},
\begin{eqnarray}\label{use-of-nerve}
  \Cech\lb G_{\text{no. }x}\rb&\simeq& \Cech\lb G_{\text{no. }x-1}\rb\vee \Sigma \link_{\Cech\lb G_{\text{no. }x}\rb}(x)\nonumber\\
  &\simeq&\Cech\lb G_{\text{no. }x-1}\rb\vee \Sigma \bigvee_{\binom{n}{2}}\S 1\nonumber\\
  &=&\Cech\lb G_{\text{no. }x-1}\rb\vee \bigvee_{\binom{n}{2}}\S 2
\end{eqnarray}
Using induction, for any $m,n\in\naturals$ and $a\in\cF_n^m$,
\begin{eqnarray}
  \Cech\lb \cF_{\preccurlyeq a}^m\rb&=&\bigvee_{\binom{m}{2}}{\binom{2}{2}}\S 2
  \vee \bigvee_{{\binom{m}{3}}{\binom{3}{2}}}\S 2 \vee \dots
  \vee \bigvee_{{\binom{m}{n-1}}{\binom{n-1}{2}}}\S 2 \vee
  \bigvee_{\#\lc b\in \cF_n^m\given b\preccurlyeq a\rc}\S 2\nonumber\\
  &=&\bigvee_{\sum_{i=2}^{n-1}\binom{m}{i}\binom{i}{2}+\#\lc b\in \cF_n^m\given b\preccurlyeq a\rc}\S 2.
\end{eqnarray}
If we take $a$ to be the last element (with CRL ordering) of $\cF_n^m$, $a\dash$ of $\cF_{n-1}^m$, then we get $\Cech\lb \cF_{\preccurlyeq a}^m\rb\setminus \Cech\lb \cF_{\preccurlyeq a\dash}^m\rb=\Cech\lb \cF_n^m\rb$.
\begin{note}
The homotopy type remains same for $2\leq r<4$ as in $r=2$.
\end{note}

\section{For general $r$}

\subsection{Calculations using Python program}

\begin{center}
    \begin{tabular}{|c|c|}
        \hline
        Complex & $4\leq r<8$ \\
        \hline
        $\Cech\lb \cF_1^4; r\rb$&$\bigvee_3\S 0$ \\
        $\Cech\lb \cF_2^4; r\rb$&$\S4$ \\
        $\Cech\lb \cF_3^4; r\rb$&$*$ \\
        \hline
        $\Cech\lb \cF_2^5; r\rb$&$\S3\vee\lb \bigvee_5\S4\rb$ \\
        $\Cech\lb \cF_3^5; r\rb$&$\S3\vee\lb\bigvee_3\S4\rb$\\
        \hline
        $\Cech\lb \cF_2^6; r\rb$&$\lb \bigvee_5\S3\rb\vee\lb\bigvee_{15}\S4\rb$\\
        $\Cech\lb \cF_3^6; r\rb$&$\lb \bigvee_{11}\S3\rb\vee\lb\bigvee_{30}\S4\rb$\\
        $\Cech\lb \cF_4^6; r\rb$&$\lb \bigvee_5\S3\rb\vee\lb\bigvee_{15}\S4\rb$\\
        \hline
        $\Cech\lb \cF_2^7; r\rb$&$\lb \bigvee_{15}\S3\rb\vee\lb\bigvee_{35}\S4\rb$\\
        $\Cech\lb \cF_3^7; r\rb$&$\lb \bigvee_{50}\S3\rb\vee\lb\bigvee_{105}\S4\rb$\\
        \hline
        $\Cech\lb \cF_3^8; r\rb$&$\lb \bigvee_{155}\S3\rb\vee\lb\bigvee_{280}\S4\rb$\\
        \hline
        $\Cech\lb \cF_3^9; r\rb$&$\lb \bigvee_{385}\S3\rb\vee\lb\bigvee_{630}\S4\rb$\\
        \hline
    \end{tabular}
\end{center}

\subsection{For any $m\in\naturals$, $n=2,3$}
\begin{lem}\label{contraction}
Consider $A=\lc \lc i_1<i_2<i_m\rc~:~\lc i_1<i_2\rc\in\cF_2^{m-1}\rc\cong \cF_2^{m-1}$. Clearly, $A\subset \cF_3^m$. Then $\p:\cF_3^m\tends A;~\lc i_1<i_2<i_m\rc\mapsto \lc i_1<i_2<m\rc$ is a contraction.
\end{lem}
\begin{proof}
If $A\subset X$ is closed, then $f:X\tends A$ is a contraction if $f|_A=\id_A$ and $d\lb f\lb x\rb, f\lb y\rb\rb\leq d\lb x,y\rb$. Let $x=\lb i_1, i_2, i_3\rb, y=\lb j_1, j_2, j_3\rb$. Then
$$d\lb f\lb x\rb, f\lb y\rb\rb=d\lb \lc i_1, i_2, m\rc, \lc j_1, j_2, m\rc\rb=
\begin{cases}
  4&\text{if }\lc i_1, i_2\rc\cap \lc j_1, j_2\rc=\phi;\\
  2&\text{if }\lc i_1, i_2\rc\cap \lc j_1, j_2\rc=1;\\
  0&\text{if }\lc i_1, i_2\rc=\lc j_1, j_2\rc
\end{cases}$$
and
$$d\lb x,y\rb=d\lb \lc i_1, i_2, i_3\rc, \lc j_1, j_2, j_3\rc\rb=
\begin{cases}
  6&\text{if }\lc i_1, i_2\rc\cap \lc j_1, j_2\rc=\phi\text{ and }j_1\neq j_2;\\
  4&\text{if }\lc i_1, i_2\rc\cap \lc j_1, j_2\rc=1\text{ or if } j_1=j_2\text{ but not both};\\
  2&\text{if }\lc i_1, i_2\rc=\lc j_1, j_2\rc\text{ or if }\lc i_1, i_2\rc\cap \lc j_1, j_2\rc=1,i_3=j_3;\\
  0&\text{if }\lc i_1, i_2\rc=\lc j_1, j_2\rc\text{ and }i_3=j_3.
\end{cases}$$
In any case, $d\lb f\lb x\rb, f\lb y\rb\rb\leq d\lb x,y\rb$.
\end{proof}
Then, $\H_q\lb\Cech\lb \cF_2^{m-1};r\rb\rb\hookrightarrow \H_q\lb\Cech\lb \cF_3^{m};r\rb\rb$ is an (injective) inclusion.
Let $\D_m\approx\Cech \lb \cF_2^m;~4\leq r<8\rb=K_1\cup K_2$ where $\D\geq K_1=\la \s_{i,m}~:~i\in\ls m-1\rs\ra$, $K_2=\la \s_{x,y}~:~1\not\in\lc x,y\rc\ra=\la\D_m\setminus\lc \s_{i,m}~:~i\in\ls m-1\rs\rc\ra$. Here $\s_{x,y}$ is a facet $B\ls \lc x,y\rc;~2\leq \dfrac{r}{2}<8\rs$.

We calculated the following homotopy types using computer:

\begin{center}
    \begin{tabular}{|c|c|c|c|c|}
    \hline
Complex & $K_1$ & $K_2$ & $K_1\cap K_2$ & Homotopy type\\
\hline
$\D_4$ & $*$ & $*$ & $\S3$ & $\S4$\\
$\D_5$ & $*$ & $\S4$ & $\S2\vee\lb \bigvee_4\S3\rb$ & $\S3\vee\lb \bigvee_5\S4\rb$\\
$\D_6$ & $*$ & $\S3\vee\lb \bigvee_5\S4\rb$ & $\lb \bigvee_4\rb\S2\vee\lb \bigvee_{10}\S3\rb$ & $\lb \bigvee_5\rb\S3\vee\lb \bigvee_{15}\S4\rb$\\
$\D_7$ & $*$ & $\lb \bigvee_5\S3\rb\vee\lb \bigvee_{15}\S4\rb$ & $\lb \bigvee_{10}\S2\rb\vee\lb \bigvee_{20}\S3\rb$ & $\lb \bigvee_{15}\S3\rb\vee\lb \bigvee_{35}\S4\rb$\\
\hline
    \end{tabular}
  \end{center}

Now, we propose the following conjecture:
\begin{conj}\label{i2-contractible}
The inclusion $i_2:K_1\cap K_2\inclusion K_2$ is null-homotopic, i.e., $K_1\cap K_2$ is contractible inside $K_2$.
\end{conj}
Note that $K_1$ is contractible since vertex $\lc 1,m\rc$ appears in every facet, and $K_1$ is a cone with apex $\lc 1,m\rc$. For $K_2$, the initial set is $\ls m-1\rs$ to choose the elements in the vertices from, so it is like $\D_{m-1}$ with some additional 2-distance-edges, e.g., there are edges from $\lc 1,m\rc$ and $\lc 2,m\rc$ to $\lc 1,2\rc$ inside $K_2$; and $\lc 1,m\rc$ and $\lc 2,m\rc$ have an edge between them. Using Lemma \ref{k1-union-k2}, $\N\lb\text{facets of }K_2\rb\simeq \lb\D_{m-1}\rb$. If the Conjecture \ref{i2-contractible} is true, then
\begin{eqnarray*}
  \D_m&\simeq&K_1\vee K_2\vee \Sigma\lb K_1\cap K_2\rb\\
  &\simeq& *\vee \D_{m-1}\vee \Sigma\lb\lb \bigvee_{\binom{m-1}{4}-\binom{m-2}{4}}\S 2\rb\vee\lb\bigvee_{\binom{m}{4}-\binom{m-1}{4}}\S 3\rb\rb\\
  &\simeq&\D_{m-1}\vee\lb\lb \bigvee_{\binom{m-2}{3}}\S 2\rb\vee \lb\bigvee_{\binom{m-1}{3}}\S 3\rb\rb
\end{eqnarray*}
This proves, using induction,
\begin{equation}
  \Cech\lb \cF_2^m;4\leq r<8\rb\simeq\lb\bigvee_{\binom{m-1}{4}}\S 3\rb\vee\lb\bigvee_{\binom{m}{4}}\S 4\rb.
\end{equation}
Further, we conjecture

\begin{conj}\label{F-3}
The following homotopy type:
\begin{equation}
\Cech\lb \cF_3^m;4\leq r<8\rb\simeq\lb\bigvee_{\g_m}\S 3\rb\vee\lb\bigvee_{\lb m-4\rb\a_m}\S 4\rb  
\end{equation}
where $\displaystyle\a_m=\binom{m}{4}$, $\displaystyle\g_m=\lb m-4\rb\a_m-\dfrac{3m+1}{4}\binom{m-2}{3}$.
\end{conj}

%\nocite{*}
 
\end{document}